\documentclass[11pt,a4paper]{article}
\usepackage[T1]{fontenc}
\usepackage{lmodern}
\usepackage[margin=27mm]{geometry}
\usepackage{amsmath,amssymb,amsthm,mathtools}
\usepackage{microtype}
\usepackage[hidelinks]{hyperref}
\newtheorem{theorem}{Theorem}[section]
\newtheorem{proposition}[theorem]{Proposition}
\newtheorem{lemma}[theorem]{Lemma}
\newtheorem*{quotedtheorem}{Theorem}
\newtheorem{corollary}[theorem]{Corollary}
\theoremstyle{remark}
\newtheorem{remark}{Remark}[section]
\newcommand{\EE}{\mathbb E}
\newcommand{\Prob}{\mathbb P}
\newcommand{\RR}{\mathbb R}
\newcommand{\Ncal}{\mathcal N}

\DeclareMathOperator{\Cov}{Cov}
\DeclareMathOperator{\diag}{diag}
\newcommand{\norm}[1]{\left\lVert #1\right\rVert}
\newcommand{\abs}[1]{\left|#1\right|}
\hypersetup{
 pdftitle={Lyapunov exponents for products of truncated orthogonal matrices},
 pdfauthor={Qichao Dong}
}
\title{Lyapunov exponents for products\\of truncated orthogonal matrices}
\author{Qichao Dong\\[4pt]
 \small School of Mathematical Sciences\\
 \small University of Science and Technology of China\\
 \small Hefei, Anhui 230026, China\\
 \small\texttt{dqccha@mail.ustc.edu.cn}}
\date{}
\begin{document}
\maketitle

\begin{abstract}
We study the Lyapunov exponents for products
of $N$ independent truncated Haar orthogonal matrices of size $n$. We prove that, in principle, provided that $N/n \to \infty$ , the first $k$ Lyapunov exponents are approximately independent and Gaussian in a non-asymptotic sense, where $k$ may depend on $N$ and $n$.
\end{abstract}


\section{Introduction}
\subsection{Model and main results}
Let \(R_i\) be independent Haar-distributed real orthogonal matrices
of size \((n+l_i)\times(n+l_i)\), and let \(A_i\) be the upper-left
\(n\times n\) block of \(R_i\), where \(l_i\ge1\) are integers.
We consider the random matrix product
\begin{equation}\label{definition}
 X_{N,n}=A_N\cdots A_1.
\end{equation}
Let \(s_1\ge\cdots\ge s_n>0\) be its singular values. The
 Lyapunov exponents are defined as 
\begin{equation}\label{eq:lambda}
 \lambda_j=\lambda_j(X_{N,n}):=\frac1N\log s_j(X_{N,n}),
 \qquad 1\le j\le n.
\end{equation}
We prove a joint Gaussian approximation with an explicit error bound
for
\[
 \Lambda_k=(\lambda_1,\ldots,\lambda_k)^{\mathsf T},\qquad 1\le k\le n.
\]
The bound allows both \(n\) and \(k\) to grow with \(N\).

We use the following notation. For every positive
integer \(m\), define
\begin{align}
 \mu_{m,l_i}
 &=\mathbb{E}\left[\log \left(\operatorname{Beta}(m/2, l_i/2)\right)\right]=\psi(m/2)-\psi((m+l_i)/2),\label{eq:mu-one}\\
 \sigma_{m,l_i}
 &=\mathrm{Var}[\log \left(\operatorname{Beta}(m/2, l_i/2)\right)]=\psi_1(m/2)-\psi_1((m+l_i)/2),\label{eq:sigma-one}
\end{align}
where \(\psi=\Gamma'/\Gamma\) and \(\psi_1=\psi'\). 
For the Gaussian approximation, let
\begin{align}
 \mu_{n,\boldsymbol l,j}&=\frac1{2N}\sum_{i=1}^N\mu_{n-j+1,l_i},&
 \sigma_{N,n,j}&=\sum_{i=1}^N\sigma_{n-j+1,l_i},\label{eq:sum-parameters}\\
 \mu_{n,\boldsymbol l}
 &=(\mu_{n,\boldsymbol l,1},\ldots,\mu_{n,\boldsymbol l,k})^{\mathsf T},&
 \Sigma_{n,\boldsymbol l}
 &=\frac1{4N^2}\diag(\sigma_{N,n,1},\ldots,\sigma_{N,n,k}).
 \label{eq:comparison-covariance}
\end{align}
Here \(\boldsymbol l=(l_1,\ldots,l_N)\). We also use the partial sums
\[
 S_j=\lambda_1+\cdots+\lambda_j,\qquad
 \mu_j=\sum_{h=1}^j\mu_{n,\boldsymbol l,h},\qquad 1\le j\le k.
\]

 For random vectors in \(\RR^k\), write Kolmogorov-Smirnov distance as 
\begin{equation}\label{eq:distance}
 d(X,Y)=\sup_{D\in\mathcal C_k}
       \abs{\Prob(X\in D)-\Prob(Y\in D)},
\end{equation}
where \(\mathcal C_k\) is the set of convex Borel subsets of
\(\RR^k\). The notation \(d(X,\Ncal(\mu,\Sigma))\) means the
distance from \(X\) to a Gaussian vector with these parameters.
This distance is unchanged by a common invertible affine map.

\begin{theorem}\label{1.1}
For the product in \eqref{definition}, there is a universal
constant \(C>0\) such that $\lambda_1,\ldots,\lambda_k$ are approximately independent and Gaussian when N is sufficiently large as a function of $k,n,l$:
\begin{equation}\label{eq:joint-bound}
 d\bigl(\Lambda_k,\Ncal(\mu_{n,\boldsymbol l},\Sigma_{n,\boldsymbol l})\bigr)
 \le{} C k^{\frac{7}{4}}\sqrt{\frac{n(n+l)}{Nl}}
 \left[\log(en)+
 \log\left(e+\sqrt{\frac{Nl}{n(n+l)}}\right)\right],
\end{equation}
where \(l=\min_{1\le i\le N}l_i\).
\end{theorem}

Our results can be directly extended to truncated Haar unitary matrices and rectangular truncated matrices products in Section~\ref{sec:extensions}. Two direct remarks are given as follows.
\begin{remark}\label{rem:growth-conditions}
The parameters \(n,k,l_i\) may all vary. For growing \(k\), the distance tends to zero if
\begin{equation}\label{eq:growth}
 \frac{Nl}{k^{7/2}n(n+l)\log^2(en)}\longrightarrow\infty.
\end{equation}
 \end{remark}
\begin{remark}
For fixed \(k\), a sufficient condition is
\[
 N\gg\left(n+\frac{n^2}{l}\right)\log^2(en).
\]
Thus \(N\gg n^2\log^2(en)\) suffices for fixed \(l\), while
\(N\gg n\log^2(en)\) suffices when \(l\ge cn\) for a fixed
\(c>0\). For fixed \(n\), \(N\to\infty\) alone suffices.
\end{remark}

\subsection{Prior Work}

Products of random matrices have attracted wide interest because of
their connections with wireless communications~\cite{TulinoVerdu04},
statistical physics and chaotic dynamical systems~\cite{Crisanti93},
free probability theory~\cite{MingoSpeicher17} and deep learning theory~\cite{HaninNica20}.
The classical theory concerns the setting where \(n\) is fixed
and \(N\) tends to infinity. This line of work dates back to Bellman~\cite{Bellman54} in 1954. 
Furstenberg and Kesten~\cite{Fursten} later
proved almost sure convergence of the largest Lyapunov exponent
under certain assumptions in 1960. Oseledec~\cite{Ose} proved the
multiplicative ergodic theorem, which gives the corresponding limits
for the full Lyapunov spectrum. Cohen and Newman~\cite{Newman}
studied the stability of large random matrices and their products.
Le Page~\cite{Le} and Carmona~\cite{Car} established central limit theorems under suitable
moment and irreducibility assumptions.

The regime where $N$ is large but fixed and $n$ tends towards infinity has also received much attention. In free probability , the spectral properties of products of independent random matrices have been extensively studied~\cite{NS,VDN}.
Akemann and Burda obtained
determinantal formulas for the complex eigenvalues of
products of Ginibre matrices~\cite{AB12}. Akemann, Burda, Kieburg, and
Nagao developed the corresponding theory for products of
truncated unitary matrices. These formulas give access to
correlation functions and their large-dimensional
limits~\cite{ABKN14}. For singular values, Kieburg,
Kuijlaars, and Stivigny established polynomial-ensemble
formulas under suitable dimension assumptions. They derived
joint densities and double contour integral representations
for the correlation kernels of products of truncated unitary
matrices, including rectangular factors~\cite{KKS16}. For complex Ginibre products, Liu, Wang, and Zhang proved
sine-kernel limits in the bulk and Airy-kernel limits at the
soft edge in this regime~\cite{LWZ16}.  

When both the matrix size and the number of factors grow,
their ratio $n/N$ becomes important for local statistics.
Akemann, Burda, and Kieburg studied Gaussian asymptotics for complex Ginibre products~\cite{ABK14} and transitions in their local statistics 
~\cite{ABK19}. Liu,
Wang, and Wang proved a three-regime transition. \(N/n\to0\) gives sine-kernel
statistics in the bulk and Airy-kernel statistics at the
soft edge, while \(N/n\to\infty\) gives Gaussian
fluctuations. When \(N/n\) tends to a positive finite
constant, new critical kernels appear~\cite{LWW}. For products
of rectangular truncated unitary matrices, Gu and Liu
obtained an analogous transition, with the relevant
parameter depending on the dimensions and truncations
of the factors~\cite{GuLiu25}. Furthermore, in the regime where $n,N \to \infty$ and $n/N$ is fixed, Gorin and Sun proved convergence
of singular values around the triangular law to Gaussian fields for
several right-unitarily invariant ensembles, including
Ginibre and Jacobi products~\cite{GorinSun22}. Ahn~\cite{Ahn19}
proved Gaussian global fluctuations for \(\beta\)-Jacobi product processes
with a fixed number of factors and studied the largest singular values when
\(n,N\to\infty\) and \(n/N\) tends to a finite positive constant.
Products of many large random matrices also arise in the study of
input-output Jacobians of randomly initialized neural
networks~\cite{PSG17,PSG18}.
The singular values of these Jacobians describe how gradients grow
or shrink across layers~\cite{Hanin18,HaninNica20}.
 
In addition to the aforementioned literature, there are also some non-asymptotics results. Hanin and Paouris developed a non-asymptotic approach
for real Gaussian products. They obtained concentration
estimates for sums of Lyapunov exponents, quantitative
convergence of the global singular-value distribution,
and a joint Gaussian approximation for the leading
exponents~\cite[Theorems~1.1--1.3]{hahn}.
 More recently, Hanin and Jiang proved a quantitative
global universality result for products with independent
entries under certain assumptions~\cite{HaninJiang25}.
Also, a Gaussian approximation for the largest
exponent of rotationally invariant products was
obtained in~\cite{Dong26}. For further open problems in this area,
we refer the readers to Deift~\cite{Deift17}.

\subsection{Outline of remainder of the article}
The rest of this article is organized as follows. Section~\ref{sec:gaussian} recalls several results needed to prove Theorem~\ref{1.1} firstly.
In Section~\ref{subsec:log-beta}, we use QR decomposition to write
fixed-frame logarithmic volumes as sums of independent log-Beta variables.
Section~\ref{subsec:tail-moments} gives tail bounds and centered moment
estimates for these variables. In Section~\ref{subsec:fixed-frame},
Lata\l a's theorem and Markov's inequality yield concentration bounds
for these sums. Section~\ref{subsec:joint-gaussian} combines the Gaussian
approximation for the independent sums with a projection estimate to
prove Theorem~\ref{1.1}. Section~\ref{sec:extensions} treats products of truncated Haar
unitary matrices and products of rectangular truncated Haar matrices.

\medskip
\noindent\textbf{Notation.} Throughout the paper, \(c\) and \(C\) denote
positive universal constants whose values may change from line to line.

\section{Proof of Theorem~\ref{1.1}}\label{sec:gaussian}
We draws on the non-asymptotic approach
of \cite{hahn} for truncated Haar orthogonal
products. For a fixed orthonormal frame, the logarithmic volume can be
written as a sum of independent log-Beta variables. 
We use the independent log-Beta
representation together with moment estimates and
a projection comparison. The resulting bound gives
the growth condition in~\eqref{eq:growth}.

Let \(\operatorname{Fr}_{n,k}\) be the set of orthonormal
\(k\)-frames in \(\RR^n\). For
\(\Theta=(\theta_1,\ldots,\theta_k)\), write
\[
 X(\Theta)=X\theta_1\wedge\cdots\wedge X\theta_k,\qquad
 \norm{X(\Theta)}^2=\det((X\Theta)^{\mathsf T}X\Theta).
\]
In particular,
\(\sup_{\Theta\in\operatorname{Fr}_{n,k}}N^{-1}\log\norm{X_{N,n}(\Theta)}=S_k\).
Small-ball estimates for Haar projections control the difference
between this fixed-frame logarithmic volume and \(S_k\). We use the following
projection and comparison estimates.

\begin{proposition}[{\cite[Proposition~8.1]{hahn}}]\label{prop:projection}
There is a universal \(C_0\ge2\) such that, for any fixed
\(\Theta\in\operatorname{Fr}_{n,k}\) and \(0<\varepsilon<1\),
\begin{equation}\label{eq:projection}
 \Prob\left(S_k-\frac1N\log\norm{X_{N,n}(\Theta)}
       \ge\frac k{2N}\log\frac{n}{k\varepsilon^2}\right)
       \le(C_0\varepsilon)^{k/2}.
\end{equation}
\end{proposition}
The estimate follows from right orthogonal invariance and a
small-ball bound for Haar projections. Since \(X_{N,n}\) is right
orthogonally invariant, it applies here.

\noindent\begin{minipage}{\linewidth}
We also use the following comparison estimate and central limit theorem.
\begin{proposition}[{\cite[Proposition~6.3]{hahn}}]\label{6.3}
There is a universal constant \(C>0\) with the following property.
Let \(X,Y\) be \(\RR^k\)-valued random variables on the same
probability space. For every \(\mu\in\RR^k\), positive definite
matrix \(\Sigma\), and \(\delta>0\), we have

\begin{equation}\label{eq:perturbation}
 d(X+Y,\Ncal(\mu,\Sigma))
 \le{}3d(X,\Ncal(\mu,\Sigma))
 +C\delta\norm{\Sigma^{-1}}_{\mathrm{HS}}^{1/2}
 +2\Prob(\norm Y_2>\delta).
\end{equation}
Here
\(\norm M_{\mathrm{HS}}=(\operatorname{tr}(M^{\mathsf T}M))^{1/2}\).
\end{proposition}
\end{minipage}

\begin{theorem}[{\cite[Theorem~1.1]{Ben05}}]\label{1.5}
Let \(X_1,\ldots,X_N\) be independent centered vectors in \(\RR^k\).
If \(\Sigma=\Cov(\sum_iX_i)\) is positive definite, then
\begin{equation}\label{eq:bentkus}
 d\left(\sum_iX_i,\Ncal(0,\Sigma)\right)
 \le Ck^{1/4}\sum_i\EE\left[\norm{\Sigma^{-1/2}X_i}_2^3\right],
\end{equation}
where \(C>0\) is universal.
\end{theorem}

Section~\ref{sec:gaussian} contains two estimates for the log-Beta
sums. The moment bounds, together with Lata\l a's theorem and
Markov's inequality, give concentration of fixed-frame volumes.
The joint Gaussian approximation uses the third absolute moment,
Bentkus's theorem, and the three-term comparison in
\cite[Section~12]{hahn}.

\subsection{Independent log-Beta increments}\label{subsec:log-beta}
Each \(A_i\) is right orthogonally invariant, so
\cite[Lemma~8.2 and Proposition~8.1]{hahn} apply to this product.
Following \cite[Lemma~9.5]{hahn}, we first decompose the
logarithmic volume for a fixed frame. We have
\begin{align}
\log \left\|X_{N, n}(\Theta)\right\| & =\log \left\|A_N \cdots A_1(\Theta)\right\| \\
& =\log \left\|A_N \cdots A_2 \frac{A_1(\Theta)}{\left\|A_1(\Theta)\right\|}\right\|+\log \left\|A_1(\Theta)\right\|.
\end{align}
Moreover, \(A_2(A_1(\Theta)/\norm{A_1(\Theta)})\) is
independent of \(A_3,\ldots,A_N\), and right orthogonal invariance gives 
\begin{equation}
	A_2\left(\frac{A_1(\Theta)}{\left\|A_1(\Theta)\right\|}\right) \stackrel{d}{=} A_2(\Theta),
\end{equation}
Repeating this argument over the layers gives
\begin{equation}
	\log \left\|X_{N, n}(\Theta)\right\| \stackrel{d}{=} \sum_{i=1}^N \log \left\|A_i(\Theta)\right\|
\end{equation}
 A column of a Haar
matrix is a normalized Gaussian vector~\cite{Dia}. For a fixed unit vector \(\theta\), if
\(g_1,\ldots,g_{n+l_i}\) are independent standard Gaussians,
\[
 \norm{A_i\theta}^2\ \stackrel{\mathrm d}=
 \frac{g_1^2+\cdots+g_n^2}{g_1^2+\cdots+g_{n+l_i}^2}
       \sim\operatorname{Beta}(n/2,l_i/2).
\]
Recall that if \(X\sim\operatorname{Gamma}(\alpha,\theta)\)
and \(Y\sim\operatorname{Gamma}(\beta,\theta)\) are independent
with a common scale parameter \(\theta\), then
\(X/(X+Y)\sim\operatorname{Beta}(\alpha,\beta)\).
For the joint distribution we apply QR decomposition to the
leading columns together.

Independent Beta laws for the squared QR diagonal entries of
truncated Haar matrices are stated by Reddy~\cite[Section~4]{Reddy}.
Forrester~\cite[Propositions~3.1--3.2]{Forrester15} obtained the
corresponding Lyapunov exponent and variance formulas.
We give a self-contained derivation of the log-Beta increments
below, proving independence of the entire array for varying
truncations.

\begin{lemma}\label{lem:qr-beta}
Fix \(n,N\ge1\), \(1\le k\le n\), and an orthonormal \(k\)-frame
\(\Theta=(\theta_1,\ldots,\theta_k)\in\operatorname{Fr}_{n,k}\).
We have
\begin{equation}\label{eq:nested-volumes}
 \log\norm{X_{N,n}(\Theta)}
 \ \stackrel{\mathrm d}{=}\ \frac12\sum_{i=1}^N\sum_{j=1}^k T_{ij},
\end{equation}
where \(\stackrel{\mathrm d}{=}\) denotes equality in distribution and
the \(T_{ij}\) are independent with
\begin{equation}\label{eq:increment-law}
 T_{ij}\sim\log\operatorname{Beta}\left(\frac{n-j+1}{2},\frac{l_i}{2}\right),
 \qquad 1\le i\le N,\quad 1\le j\le k.
\end{equation}
Write \(\Theta_j=(\theta_1,\ldots,\theta_j)\) and set
\(\norm{X_{N,n}(\Theta_0)}=1\). Then, jointly for \(1\le j\le k\),
\begin{equation}\label{eq:fixed-exponents}
 \widehat\lambda_j
 :=\frac1N\log\frac{\norm{X_{N,n}(\Theta_j)}}
                         {\norm{X_{N,n}(\Theta_{j-1})}}
 \ \stackrel{\mathrm d}{=}\ \frac1{2N}\sum_{i=1}^N T_{ij}.
\end{equation}
Thus the coordinates of
\(\widehat\Lambda_k=(\widehat\lambda_1,\ldots,\widehat\lambda_k)^{\mathsf T}\)
are independent, with mean \(\mu_{n,\boldsymbol l}\) and covariance
\(\Sigma_{n,\boldsymbol l}\).
\end{lemma}

\begin{proof}
Starting from \(\Theta^{(0)}=\Theta\), take
the thin QR decompositions with positive diagonal entries
\[
 A_i\Theta^{(i-1)}=\Theta^{(i)}\mathcal R_i,\qquad 1\le i\le N.
\]
Since the matrices \(\mathcal R_i\) are upper
triangular,
\[
 X_{N,n}\Theta=\Theta^{(N)}\mathcal R_N\cdots\mathcal R_1,
 \qquad
 \norm{X_{N,n}(\Theta)}
       =\prod_{i=1}^N\prod_{j=1}^k(\mathcal R_i)_{jj}.
\]
Condition on \(A_1,\ldots,A_{i-1}\). Since \(A_i\) is independent of
the preceding matrices and is right orthogonally invariant,
\(A_i\Theta^{(i-1)}\) has the conditional law of \(A_i\Theta\).
Thus the conditional law of \(\mathcal R_i\) is the law of the
triangular factor in the QR decomposition of \(A_i\Theta\).
It does not depend on the preceding matrices, so the
\(\mathcal R_i\) are independent. The product formula therefore
writes the logarithmic volume in distribution as the sum of the
independent contributions \(\log\norm{A_i(\Theta)}\).

Let \(A\) be the upper-left block
of a Haar matrix in
\(\mathrm O(n+l)\). For a fixed frame, right invariance
reduces its QR decomposition to that of \(A(e_1,\ldots,e_k)\):
extend the fixed frame to \(V\in\mathrm O(n)\) and multiply the
full Haar matrix on the right by \(\diag(V,I_l)\).
Let \(u_1,\ldots,u_k\) be the first columns of this Haar matrix,
put \(E=\RR^n\times\{0\}\), and write \(v_j=P_Eu_j\). Define
\[
 F_j=\operatorname{span}(u_1,\ldots,u_{j-1})^\perp,
 \qquad L_{j-1}=\operatorname{span}(v_1,\ldots,v_{j-1}).
\]
For \(w\in E\), \(\langle w,u_h\rangle=\langle w,v_h\rangle\).
Hence \(E\cap F_j=E\cap L_{j-1}^\perp\), and the \(j\)th
Gram--Schmidt residual is
\begin{equation}\label{eq:qr-residual}
 v_j-P_{L_{j-1}}v_j=P_{E\cap F_j}u_j.
\end{equation}
If the first \(j-1\) truncated columns are linearly independent, then
\(\dim F_j=n+l-j+1\) and \(\dim(E\cap F_j)=n-j+1\).
Conditionally on the preceding Haar columns, \(u_j\) is uniform
on the unit sphere of \(F_j\). In an orthonormal basis whose
first \(n-j+1\) vectors span \(E\cap F_j\), the squared residual
length therefore has the distribution
\[
 \mathcal R_{jj}^2\ \stackrel{\mathrm d}=
 \frac{S}{S+T},\qquad
 S\sim\chi^2_{n-j+1},\quad T\sim\chi^2_l,
\]
with \(S,T\) independent, which gives
\begin{equation}\label{eq:qr-beta}
 \mathcal R_{jj}^2\sim\operatorname{Beta}
          \left(\frac{n-j+1}{2},\frac l2\right).
\end{equation}
The residual is nonzero almost surely, so induction proves full
column rank and the existence of the QR decomposition.
The conditional law in \eqref{eq:qr-beta} does not depend on the
preceding columns. Successive conditioning proves independence
of the diagonal squares. In particular, their joint law is the
same for every deterministic frame. Both dimensions in
\eqref{eq:qr-residual} decrease by one at each step, so the
second Beta parameter stays equal to \(l/2\).

It follows that the array of logarithms \(\log((\mathcal R_i)_{jj}^2)\)
has the joint law of the independent variables in \eqref{eq:increment-law}.
Taking logarithms of the volume product proves \eqref{eq:nested-volumes}.
For the first \(j\) columns of \(\Theta\), the volume product uses
the leading \(j\times j\) blocks of the same triangular factors.
The ratio of the volumes of \(\Theta_j\) and \(\Theta_{j-1}\) after
applying \(X_{N,n}\) is therefore \(\prod_{i=1}^N(\mathcal R_i)_{jj}\).
Taking logarithms and using the joint law of the diagonal squares
gives \eqref{eq:fixed-exponents} jointly in \(j\).
\end{proof}

\subsection{Tails and centered moments}\label{subsec:tail-moments}
We first treat \(T_i=T_{i1}\). Define
 \(n_j=n-j+1\). The same estimates will then apply
to \(T_{ij}\) with \(n\) replaced by \(n_j\).

We first recall the concentration estimate for Beta variables.
\begin{quotedtheorem}[{\cite[Theorem~1]{Sko}}]
Let \(B\sim\operatorname{Beta}(\alpha,\beta)\), where
\(\alpha,\beta>0\), and define
\[
 m=\frac{\alpha}{\alpha+\beta},\qquad
 v=\frac{\alpha\beta}{(\alpha+\beta)^2(\alpha+\beta+1)},\qquad
 c=\frac{2(\beta-\alpha)}{(\alpha+\beta)(\alpha+\beta+2)}.
\]
For every \(t>0\),
\[
 \Prob(B-m\ge t)\le
 \begin{cases}
  \displaystyle\exp\left[-\frac{t^2}{2(v+ct/3)}\right],
       &\beta\ge\alpha,\\[5pt]
  \displaystyle\exp\left[-\frac{t^2}{2v}\right],
       &\beta<\alpha.
 \end{cases}
\]
Applying this estimate to
\(1-B\sim\operatorname{Beta}(\beta,\alpha)\) gives
\[
 \Prob(m-B\ge t)\le
 \begin{cases}
  \displaystyle\exp\left[-\frac{t^2}{2(v-ct/3)}\right],
       &\alpha\ge\beta,\\[5pt]
  \displaystyle\exp\left[-\frac{t^2}{2v}\right],
       &\alpha<\beta.
 \end{cases}
\]
\end{quotedtheorem}

\begin{lemma}[Tails and centered moments]\label{lem:beta-route}
Let \(B_i\sim\operatorname{Beta}(n/2,l_i/2)\), \(T_i=\log B_i\), and
\[
 \bar T_i=T_i-\mu_{n,l_i},\qquad
 \widehat T_i=T_i-\log\frac n{n+l_i},\qquad
 w_i=\frac{2l_i}{n(n+l_i)}.
\]
For all \(s>0\) and \(p\ge1\),
\begin{align}
 \Prob(\abs{\widehat T_i}\ge s)
  &\le2\exp[-c\min\{s^2/w_i,ns\}],\label{eq:beta-route-tail}\\
 \norm{\bar T_i}_{L^p}
  &\le C\left(\sqrt{pw_i}+\frac pn\right)
   \le C\left(\sqrt{p\sigma_{n,l_i}}+\frac pn\right).
   \label{eq:logbeta-moment}
\end{align}
In the notation \(M_i=n^2(n+l_i+2)/(n+l_i)^2\), we also have
\begin{equation}\label{eq:original-M}
 \norm{\bar T_i}_{L^p}\le C\left(\sqrt{p/M_i}+p/M_i\right).
\end{equation}
\end{lemma}

\begin{proof}
We first make a reduction. We check that the estimates for
\(\bar T_i=T_i-\EE[T_i]\) follow from the corresponding estimates for
\[
 \widehat T_i=T_i-\log\frac n{n+l_i}.
\]
Set \(\varepsilon_i=\log(n/(n+l_i))-\mu_{n,l_i}\), so that
\(\bar T_i=\widehat T_i+\varepsilon_i\). The trigamma series gives
\[
 \frac1u\le\psi_1(u)=\sum_{q=0}^\infty(u+q)^{-2}
        \le\frac1u+\frac1{u^2}.
\]
Integrating from \(n/2\) to \((n+l_i)/2\), we obtain
\begin{equation}\label{eq:change-center}
 0\le\varepsilon_i=\log\frac n{n+l_i}-\mu_{n,l_i}
 \le\frac2n-\frac2{n+l_i}=w_i\le\frac2n.
\end{equation}
Suppose that the required moment bound holds for \(\widehat T_i\).
For a positive integer \(p\), the binomial formula then gives
\begin{align*}
 \EE\left[\abs{\bar T_i}^p\right]
 &\le\sum_{q=0}^p\binom pq
       \EE\left[\abs{\widehat T_i}^q\right]\varepsilon_i^{p-q}\\
 &\le\sum_{q=0}^p\binom pq
       \left[C\left(\sqrt{pw_i}+\frac pn\right)\right]^q
       \varepsilon_i^{p-q}\\
 &=\left[C\left(\sqrt{pw_i}+\frac pn\right)
                    +\varepsilon_i\right]^p.
\end{align*}
Since \(\varepsilon_i\le2/n\), taking \(p\)-th roots proves the
same bound for \(\bar T_i\). Real orders \(p\ge1\) follow by taking
the next larger integer and using monotonicity of \(L^p\) norms.
It therefore remains to estimate \(\widehat T_i\).

Write \(m_i=n/(n+l_i)\). We split the probability into the two
tails of the Beta variable,
\begin{align*}
 \Prob(\abs{\widehat T_i}\ge s)
 &=\Prob\left(\abs{\log(B_i/m_i)}\ge s\right)\\
 &=\Prob(B_i\ge m_i e^s)+\Prob(B_i\le m_i e^{-s}),\qquad s>0.
\end{align*}
Let us first bound the upper tail. The estimate of
Skorski~\cite[Theorem~1]{Sko}, applied also to \(1-B_i\), yields
\begin{equation}\label{eq:beta-additive}
 \max\{\Prob(B_i-m_i\ge t),\Prob(m_i-B_i\ge t)\}
 \le\exp\left[-\frac{t^2}{2(v_i+\eta_i t/3)}\right],
\end{equation}
where
\[
 v_i=\frac{2nl_i}{(n+l_i)^2(n+l_i+2)},\qquad
 \eta_i=\frac{4\abs{l_i-n}}{(n+l_i)(n+l_i+4)}.
\]
Here \(v_i\) is the variance of \(B_i\), and the absolute value in
\(\eta_i\) gives one bound valid for both tails. In particular,
\(v_i/m_i^2\le w_i\) and \(\eta_i/m_i\le4/n\). Thus
\begin{align*}
 \Prob(B_i\ge m_i e^s)
 &=\Prob\bigl(B_i-m_i\ge m_i(e^s-1)\bigr)\\
 &\le\exp\left[-\frac{(e^s-1)^2}
                    {2(w_i+4(e^s-1)/(3n))}\right]\\
 &\le\exp\left[-\frac{s^2}{2(w_i+4s/(3n))}\right]\\
 &\le\exp[-c\min\{s^2/w_i,ns\}],
\end{align*}
where we used \(e^s-1\ge s\) and the fact that
\(u^2/(w_i+4u/(3n))\) is increasing for \(u\ge0\).
If \(m_i e^s\ge1\), the probability is zero.

We now estimate the lower tail. For \(0<s\le2\), the inequality
\(1-e^{-s}\ge e^{-2}s\) and \eqref{eq:beta-additive} give
\begin{align*}
 \Prob(B_i\le m_i e^{-s})
 &=\Prob\bigl(m_i-B_i\ge m_i(1-e^{-s})\bigr)\\
 &\le\exp[-c\min\{s^2/w_i,ns\}].
\end{align*}
For \(s\ge2\), we use a small-ball estimate. Write
\(B_i=G_i/(G_i+H_i)\), where \(G_i,H_i\) are independent Gamma
variables with shapes \(n/2,l_i/2\) and rate \(1\). For
\(0<x<m_i\) and \(0<r<1/x\), Markov's inequality gives
\begin{align*}
 \Prob(B_i\le x)
 &=\Prob\bigl((1-x)G_i\le xH_i\bigr)\\
 &\le\EE\bigl[e^{-r(1-x)G_i}\bigr]\,\EE\bigl[e^{rxH_i}\bigr]\\
 &=(1+r(1-x))^{-n/2}(1-rx)^{-l_i/2}.
\end{align*}
Choose \(r=(m_i-x)/(x(1-x))\). Then
\begin{equation}\label{eq:beta-small-ball}
 \Prob(B_i\le x)
 \le(x/m_i)^{n/2}\bigl((1-x)/(1-m_i)\bigr)^{l_i/2}.
\end{equation}
With \(x=m_i e^{-s}\) and \(\log(1+u)\le u\), this becomes
\begin{align*}
 \log\Prob(B_i\le m_i e^{-s})
 &\le-\frac{ns}2+\frac{l_i}2
      \log\left(1+\frac n{l_i}(1-e^{-s})\right)\\
 &\le-\frac n2(s-1+e^{-s})\le-\frac{ns}4,\qquad s\ge2.
\end{align*}
Since \(nw_i<2\), the minimum in \eqref{eq:beta-route-tail}
equals \(ns\) for \(s\ge2\). Combining the two tails proves
\eqref{eq:beta-route-tail}.

To complete the moment estimate, we integrate the tail probability,
\begin{align*}
 \EE\left[\abs{\widehat T_i}^p\right]
 &=p\int_0^\infty s^{p-1}\Prob(\abs{\widehat T_i}>s)\,ds\\
 &\le2p\int_0^\infty s^{p-1}
                     (e^{-cs^2/w_i}+e^{-cns})\,ds\\
 &=2(w_i/c)^{p/2}\Gamma(p/2+1)+2(cn)^{-p}\Gamma(p+1).
\end{align*}
The bounds \(\Gamma(p/2+1)^{1/p}\le C\sqrt p\) and
\(\Gamma(p+1)^{1/p}\le Cp\), valid for \(p\ge1\), give
\[
 \left(\EE\left[\abs{\widehat T_i}^p\right]\right)^{1/p}
 \le C\left(\sqrt{pw_i}+\frac pn\right).
\]
The reduction at the start of the proof gives the same estimate
for \(\bar T_i\).

It remains to compare \(w_i\) with the variance. The function
\(g(u)=(n/2+u)^{-2}-((n+l_i)/2+u)^{-2}\) is positive and decreasing.
Its integral is \(w_i\), and \(g(0)/w_i\le4/n\). Therefore
\begin{equation}\label{eq:variance-comparison}
 w_i\le\sigma_{n,l_i}=\sum_{q=0}^\infty g(q)
       \le(1+4/n)w_i\le5w_i.
\end{equation}
This proves \eqref{eq:logbeta-moment}. Finally,
\(w_iM_i\le1\) and \(M_i\le2n\), so
\[
 \left(\EE\left[\abs{\bar T_i}^p\right]\right)^{1/p}
 \le C\left(\sqrt{\frac p{M_i}}+\frac p{M_i}\right),
\]
which is \eqref{eq:original-M}.
\end{proof}

Before applying Theorem~\ref{1.5}, we need the following
estimate on the third centered moment.
\begin{lemma}\label{lem:variance}
There exists a universal constant \(C>0\) such that, for any
positive integers \(m,l_i\),
\begin{equation}\label{eq:third-from-beta}
 \EE\left[\abs{\log B-\EE[\log B]}^3\right]\le C\sigma_{m,l_i}^{3/2},
 \qquad B\sim\operatorname{Beta}(m/2,l_i/2).
\end{equation}
Moreover,
\begin{equation}\label{eq:total-variance}
 \sigma_{N,n,j}\ge\sigma_{N,n,1}
       \ge\frac{2Nl}{n(n+l)},\qquad
 \sigma_{n_j,l_i}\le\frac{\pi^2}{2}.
\end{equation}
\end{lemma}
\begin{proof}
Equation~\eqref{eq:variance-comparison} implies
\(m^2\sigma_{m,l_i}\ge2ml_i/(m+l_i)\ge1\).
Hence \(m^{-1}\le\sqrt{\sigma_{m,l_i}}\), and the moment estimate
with \(p=3\) proves \eqref{eq:third-from-beta}.
Each term in the series for \(\sigma_{m,l_i}\) decreases as \(m\)
increases. It also increases as \(l_i\) increases. This proves
the first part of \eqref{eq:total-variance}, after summing the
lower bound in \eqref{eq:variance-comparison}. The upper bound is
\(\sigma_{n_j,l_i}\le\psi_1(1/2)=\pi^2/2\).
\end{proof}

\subsection{Concentration of fixed-frame volumes}\label{subsec:fixed-frame}
We next bound the moments of the centered log-Beta sums.
Markov's inequality will then give a concentration estimate for
the fixed-frame logarithmic volume. We use the following result
of Lata\l a.

\begin{quotedtheorem}[{\cite[Theorem~2 and Remark~2]{Lat}}]
Let \(X_1,\ldots,X_N\) be independent centered real random
variables, and let \(p\ge2\). There is a universal constant \(C>0\)
such that
\[
 \left(\EE\left[\abs{\sum_{i=1}^N X_i}^p\right]\right)^{1/p}
 \le C\inf\left\{t>0:
    \sum_{i=1}^N\log\left(
       \max_{\pm}\EE\left[\abs{1\pm X_i/t}^p\right]\right)\le p\right\}.
\]
\end{quotedtheorem}
For centered variables, \cite[Remark~2]{Lat} uses the average
of the two signs. The maximum above bounds this average.

We apply this form directly to the centered log-Beta variables.
We obtain the following estimate.

\begin{proposition}
\label{prop:concentration}
There exist universal constants \(c,C>0\) with the following property.
For a fixed \(1\le k\le n\), let
\[
 T_i=\sum_{j=1}^kT_{ij},\qquad
 \bar T_i=T_i-\EE[T_i],\qquad
 \sigma_{N,n}^{(k)}=\sum_{j=1}^k\sigma_{N,n,j}.
\]
Then, for every \(p\ge1\),
\begin{equation}\label{eq:sum-moment}
 \left(\EE\left[\abs{\sum_{i=1}^N\bar T_i}^p\right]\right)^{1/p}
 \le C\left(\sqrt{p\sigma_{N,n}^{(k)}}+\frac p{n-k+1}\right).
\end{equation}
For a fixed frame \(\Theta_k\) and every \(s>0\),
\begin{equation}\label{eq:volume-concentration}
\begin{split}
 &\Prob\left(\abs{\frac1N\log\norm{X_{N,n}(\Theta_k)}-\mu_k}>s\right)\\
 &\hspace{10mm}\le2\exp\left[-c\min\left\{
          \frac{N^2s^2}{\sigma_{N,n}^{(k)}},N(n-k+1)s\right\}\right].
\end{split}
\end{equation}
\end{proposition}
\begin{proof}
With Lemma~\ref{lem:beta-route} in hand, we are now in a position
to estimate the moments of the sum. We use the notation
\[
 \bar T_{ij}=T_{ij}-\mu_{n_j,l_i},\qquad
 p_0=(n-k+1)^2\sigma_{N,n}^{(k)}.
\]
Then \(\bar T_i=\sum_{j=1}^k\bar T_{ij}\). Since
\[
 \sqrt{p\sigma_{N,n}^{(k)}}\le\frac p{n-k+1}
       \quad\Longleftrightarrow\quad p\ge p_0,
\]
we will prove the two estimates
\[
 \left\|\sum_{i=1}^N\bar T_i\right\|_{L^p}
 \le
 \begin{cases}
  C\sqrt{p\sigma_{N,n}^{(k)}},&1\le p\le p_0,\\[3pt]
  Cp/(n-k+1),&p\ge p_0.
 \end{cases}
\]

We first take an even integer \(p\ge4\). The variables
\(\bar T_{ij}\) are independent and centered. All sums over
\(i,j\) below have \(1\le i\le N\) and \(1\le j\le k\).
By \cite[Theorem~2 and Remark~2]{Lat},
\[
 \left\|\sum_{i,j}\bar T_{ij}\right\|_{L^p}
 \le C\inf\left\{t>0:
       \sum_{i,j}\log\left(
       \max_{\pm}\EE\left[\abs{1\pm\bar T_{ij}/t}^p\right]\right)\le p\right\}.
\]
Since \(p\) is even, we can drop the absolute value and use the
binomial formula. The first-order term vanishes because
\(\EE[\bar T_{ij}]=0\). Thus, for either sign,
\begin{align*}
 \EE\left[\abs{1\pm\bar T_{ij}/t}^p\right]
 &=1+\sum_{\ell=2}^{p}\binom p\ell
       \frac{(\pm1)^\ell\EE\left[\bar T_{ij}^{\ell}\right]}{t^{\ell}}\\
 &\le1+\sum_{\ell=2}^{p}\binom p\ell
       \frac{\EE\left[\abs{\bar T_{ij}}^{\ell}\right]}{t^{\ell}}.
\end{align*}
We now bound the sum by separating the even and odd terms.

For the even terms, Lemma~\ref{lem:beta-route} gives
\[
 \sum_{\ell=1}^{p/2}\binom p{2\ell}
       \frac{\EE\left[\abs{\bar T_{ij}}^{2\ell}\right]}{t^{2\ell}}
 \le\sum_{\ell=1}^{p/2}\binom p{2\ell}
       \frac{C^{2\ell}\bigl((\ell\sigma_{n_j,l_i})^\ell
                         +(\ell/n_j)^{2\ell}\bigr)}{t^{2\ell}}.
\]
Using
\[
 \binom p{2\ell}\le\frac{p^{2\ell}}{(2\ell)!},\qquad
 \frac{\ell^\ell}{(2\ell)!}\le\frac1{\ell!},\qquad
 \frac{\ell^{2\ell}}{(2\ell)!}\le(e/2)^{2\ell},
\]
we obtain when  $ t\ge\frac{C_1p}{n-k+1},$
\begin{align*}
 \sum_{\ell=1}^{p/2}\binom p{2\ell}
       \frac{\EE\left[\abs{\bar T_{ij}}^{2\ell}\right]}{t^{2\ell}}
 &\le\exp\left(\frac{Cp^2\sigma_{n_j,l_i}}{t^2}\right)-1
       +\sum_{\ell=1}^{p/2}\left(\frac{Cp}{tn_j}\right)^{2\ell}\\
 &\le\exp\left(\frac{Cp^2\sigma_{n_j,l_i}}{t^2}\right)-1
       +\frac{(Cp/(tn_j))^2}{1-(Cp/(tn_j))^2}.
      \end{align*}
Here \(C_1\) is a sufficiently large universal constant.

For the odd terms, Cauchy--Schwarz gives, for
\(1\le\ell\le p/2-1\),
\[
 \EE\left[\abs{\bar T_{ij}}^{2\ell+1}\right]
 \le\left(\EE\left[\abs{\bar T_{ij}}^{2\ell}\right]
          \EE\left[\abs{\bar T_{ij}}^{2\ell+2}\right]\right)^{1/2}.
\]
Together with
\[
 \binom p{2\ell+1}
 \le2\left(\binom p{2\ell}\binom p{2\ell+2}\right)^{1/2},
\]
this implies
\begin{align*}
 \binom p{2\ell+1}
       \frac{\EE\left[\abs{\bar T_{ij}}^{2\ell+1}\right]}{t^{2\ell+1}}
 \le{}&\binom p{2\ell}
       \frac{\EE\left[\abs{\bar T_{ij}}^{2\ell}\right]}{t^{2\ell}}\\
 &+\binom p{2\ell+2}
       \frac{\EE\left[\abs{\bar T_{ij}}^{2\ell+2}\right]}{t^{2\ell+2}}.
\end{align*}
Summing over \(\ell\), the odd terms are bounded by twice the
sum of the even terms. Consequently,
\begin{align*}
 \max_{\pm}\EE\left[\abs{1\pm\bar T_{ij}/t}^p\right]
 &\le1+3\sum_{\ell=1}^{p/2}\binom p{2\ell}
             \frac{\EE\left[\abs{\bar T_{ij}}^{2\ell}\right]}{t^{2\ell}}\\
 &\le1+3\left[\exp\left(\frac{Cp^2\sigma_{n_j,l_i}}{t^2}\right)-1\right]
       +\frac{3(Cp/(tn_j))^2}{1-(Cp/(tn_j))^2}\\
 &\le\exp\left(\frac{C'p^2\sigma_{n_j,l_i}}{t^2}\right),
       \qquad t\ge\frac{C_1p}{n-k+1}.
\end{align*}
In the last step we used \(n_j^{-2}\le\sigma_{n_j,l_i}\), which
follows from Lemma~\ref{lem:variance}. Taking logarithms and
summing gives
\[
 \sum_{i,j}\log\left(
       \max_{\pm}\EE\left[\abs{1\pm\bar T_{ij}/t}^p\right]\right)
 \le\frac{Cp^2}{t^2}\sum_{i,j}\sigma_{n_j,l_i}
 =\frac{Cp^2\sigma_{N,n}^{(k)}}{t^2}.
\]

When \(p\le p_0\), set
\[
 t=D\sqrt{p\sigma_{N,n}^{(k)}}.
\]
Then \(t\ge Dp/(n-k+1)\), and the last sum is at most
\(Cp/D^2\). When \(p\ge p_0\), set
\[
 t=\frac{Dp}{n-k+1}.
\]
The same sum is at most \(Cp_0/D^2\le Cp/D^2\).
Taking \(D\) sufficiently large makes the sum at most \(p\)
in both cases. Lata\l a's estimate therefore proves the claimed
moment bounds for even \(p\ge4\).

For real \(p>2\), let \(p'\) be the smallest even integer at least
\(p\). Then \(p'\le2p\), so monotonicity of \(L^p\) norms proves
\eqref{eq:sum-moment}, with a change of the universal constant.
For \(1\le p\le2\), independence gives
\[
 \left\|\sum_{i,j}\bar T_{ij}\right\|_{L^p}
 \le\left\|\sum_{i,j}\bar T_{ij}\right\|_{L^2}
 =\sqrt{\sigma_{N,n}^{(k)}},
\]
which proves the remaining cases.

For \(k=1\), if \(M=\min_{i\le N}M_i\), then
\(\sigma_{n,l_i}\le5/M_i\) and \(1/n\le2/M\). Consequently,
\[
 \left(\EE\left[\abs{\sum_{i=1}^N\bar T_i}^p\right]\right)^{1/p}
 \le C\left(\sqrt{p\sum_{i=1}^N M_i^{-1}}+\frac pM\right)
 \le C\left(\sqrt{\frac{pN}{M}}+\frac pM\right).
\]

We now obtain the concentration estimate. By Lemma~\ref{lem:qr-beta},
\[
 \frac1N\log\norm{X_{N,n}(\Theta_k)}-\mu_k
 \ \stackrel{\mathrm d}{=}\ \frac1{2N}\sum_{i=1}^N\bar T_i.
\]
Since \(\sigma_{n-k+1,l_i}\ge(n-k+1)^{-2}\), we have
\(p_0\ge N\ge1\). Markov's inequality and \eqref{eq:sum-moment}
give, for a sufficiently large universal constant \(C\ge1\),
\[
 \Prob\left(\abs{\frac1N\log\norm{X_{N,n}(\Theta_k)}-\mu_k}
    >\frac C{N}\sqrt{p\sigma_{N,n}^{(k)}}\right)\le e^{-p},
 \qquad 1\le p\le p_0,
\]
and
\[
 \Prob\left(\abs{\frac1N\log\norm{X_{N,n}(\Theta_k)}-\mu_k}
    >\frac{Cp}{N(n-k+1)}\right)\le e^{-p},
 \qquad p\ge p_0.
\]

First suppose that
\[
 0<s\le\frac{C(n-k+1)\sigma_{N,n}^{(k)}}N.
\]
If \(s\ge C\sqrt{\sigma_{N,n}^{(k)}}/N\), take
\[
 p=\frac{N^2s^2}{C^2\sigma_{N,n}^{(k)}}\in[1,p_0]
\]
in the first estimate. For \(0<s<C\sqrt{\sigma_{N,n}^{(k)}}/N\),
we have
\[
 2\exp\left[-\frac{\log2}{C^2}
                  \frac{N^2s^2}{\sigma_{N,n}^{(k)}}\right]\ge1.
\]
Thus, throughout this range of \(s\),
\[
 \Prob\left(\abs{\frac1N\log\norm{X_{N,n}(\Theta_k)}-\mu_k}>s\right)
 \le2\exp\left[-c\frac{N^2s^2}{\sigma_{N,n}^{(k)}}\right],
 \qquad c=\frac{\log2}{C^2}.
\]

Now suppose that \(s\ge C(n-k+1)\sigma_{N,n}^{(k)}/N\). Taking
\[
 p=\frac{N(n-k+1)s}C\ge p_0\ge1
\]
in the second estimate gives
\[
 \Prob\left(\abs{\frac1N\log\norm{X_{N,n}(\Theta_k)}-\mu_k}>s\right)
 \le\exp\left[-\frac{N(n-k+1)s}C\right].
\]
Since \(c\le1/C\), combining the two ranges proves
\eqref{eq:volume-concentration}.
\end{proof}

\subsection{Completion of proof of Theorem~\ref{1.1}}\label{subsec:joint-gaussian}
We now prove Theorem~\ref{1.1}. We follow the proof of
\cite[Theorem~1.3]{hahn} and divide the estimate into three parts.
Fix the frame in Lemma~\ref{lem:qr-beta} and, for \(1\le j\le k\), set
\[
 \widehat S_j=\frac1N\log\norm{X_{N,n}(\Theta_j)},\qquad
 \boldsymbol S_k=(S_1,\ldots,S_k)^{\mathsf T},\qquad
 \widehat{\boldsymbol S}_k=(\widehat S_1,\ldots,\widehat S_k)^{\mathsf T}.
\]
Let \(T\) be
the lower triangular \(k\times k\) matrix whose entries on and
below the diagonal are one. Then
\[
 \boldsymbol S_k=T\Lambda_k,\qquad
 \widehat{\boldsymbol S}_k=T\widehat\Lambda_k,\qquad
 \Gamma_N=T\Sigma_{n,\boldsymbol l}T^{\mathsf T}.
\]
Proposition~\ref{6.3} and affine invariance give
\begin{equation}\label{eq:three-terms}
 d(\Lambda_k,\Ncal(\mu_{n,\boldsymbol l},\Sigma_{n,\boldsymbol l}))
 \le3d(\widehat\Lambda_k,
 \Ncal(\mu_{n,\boldsymbol l},\Sigma_{n,\boldsymbol l}))
 +C\delta\norm{\Gamma_N^{-1}}_{\mathrm{HS}}^{1/2}
 +2\Prob(\norm{\boldsymbol S_k-\widehat{\boldsymbol S}_k}_2>\delta).
\end{equation}
For part 1, we use Theorem~\ref{1.5} and the third-moment
estimate in Lemma~\ref{lem:variance}.
\begin{lemma}\label{lem:fixed-clt}
There exists a universal constant \(C>0\) such that, for any
\(n,N\ge1\) and \(1\le k\le n\), if
\[
 \rho_N=\max_{i\le N,j\le k}
              \frac{\sigma_{n_j,l_i}}{\sigma_{N,n,j}},
\]
then
\begin{equation}\label{eq:fixed-clt}
 d(\widehat\Lambda_k,\Ncal(\mu_{n,\boldsymbol l},\Sigma_{n,\boldsymbol l}))
 \le Ck^{7/4}\sqrt{\rho_N}.
\end{equation}
If \(l_i\) are all equal, \(\rho_N=1/N\).
\end{lemma}
\begin{proof}
For part 1, we follow the notation in Bentkus~\cite{Ben05}.
By Lemma~\ref{lem:qr-beta},
\[
 \widehat\Lambda_k\ \stackrel{\mathrm d}{=}\ \frac1N\sum_{i=1}^N Y_i,\qquad
 Y_i=\frac12(T_{i1},\ldots,T_{ik})^{\mathsf T}.
\]
The vectors \(Y_i\) are independent. Set
\[
 \bar Y_i=Y_i-\EE[Y_i],\qquad
 \Sigma=\Cov\left(\sum_{i=1}^N\bar Y_i\right)
     =N^2\Sigma_{n,\boldsymbol l},
\]
and define
\[
 \beta_i=\EE\left[\norm{\Sigma^{-1/2}\bar Y_i}_2^3\right],\qquad
 \beta=\sum_{i=1}^N\beta_i.
\]
We have
\[
 (\Sigma^{-1/2}\bar Y_i)_j
   =\frac{T_{ij}-\mu_{n_j,l_i}}{\sqrt{\sigma_{N,n,j}}}.
\]
Define \(q_i=\sum_{j=1}^k\sigma_{n_j,l_i}/\sigma_{N,n,j}\).
Minkowski's inequality in \(L^{3/2}\) and
Lemma~\ref{lem:variance} give
\begin{align*}
 \beta_i^{2/3}
 &=\left(\EE\left[\left(\sum_{j=1}^k
       \frac{\abs{T_{ij}-\mu_{n_j,l_i}}^2}{\sigma_{N,n,j}}
                    \right)^{3/2}\right]\right)^{2/3}\\
 &\le\sum_{j=1}^k
       \frac{\left(\EE\left[\abs{T_{ij}-\mu_{n_j,l_i}}^3\right]\right)^{2/3}}
            {\sigma_{N,n,j}}
 \le C_1q_i.
\end{align*}
Since \(\sum_iq_i=k\) and \(q_i\le k\rho_N\), it follows that
\[
 \beta\le C_1\sum_{i=1}^N q_i^{3/2}
       \le C_1k^{3/2}\sqrt{\rho_N}.
\]
Applying Theorem~\ref{1.5}, we conclude that
\[
 d\bigl(\widehat\Lambda_k,
              \Ncal(\mu_{n,\boldsymbol l},\Sigma_{n,\boldsymbol l})\bigr)
 \le C_1k^{1/4}\beta
 \le C_1k^{7/4}\sqrt{\rho_N}.
\]
This proves \eqref{eq:fixed-clt}.
\end{proof}

For part 2, we estimate the inverse covariance matrix.
\begin{lemma}\label{lem:inverse}
For any \(n,N\ge1\) and \(1\le k\le n\), the matrix
\(\Gamma_N=T\Sigma_{n,\boldsymbol l}T^{\mathsf T}\) satisfies
\begin{equation}\label{eq:inverse}
 \norm{\Gamma_N^{-1}}_{\mathrm{HS}}^{1/2}
 \le4N\left(\sum_{j=1}^k\sigma_{N,n,j}^{-2}\right)^{1/4}
 \le\frac{4Nk^{1/4}}{\sqrt{\sigma_{N,n,1}}}.
\end{equation}
\end{lemma}
\begin{proof}
Let \(B=T^{-1}\). The matrix \(B\) takes differences:
\((Bx)_j=x_j-x_{j-1}\), where \(x_0=0\), so
\(\norm B_{\mathrm{op}}\le2\). From the factorization of
\(\Gamma_N\),
\[
 \Gamma_N^{-1}=4N^2B^{\mathsf T}
       \diag(\sigma_{N,n,1}^{-1},\ldots,\sigma_{N,n,k}^{-1})B.
\]
This matrix is tridiagonal. If \(d_j=\sigma_{N,n,j}^{-1}\), its
diagonal entries, after division by \(4N^2\), are
\(d_j+d_{j+1}\) for \(j<k\), and \(d_k\) for \(j=k\).
Its adjacent off-diagonal entries are \(-d_{j+1}\).
The Hilbert--Schmidt norm is therefore at most
\[
 4N^2\norm B_{\mathrm{op}}^2(\sum_jd_j^2)^{1/2}
       \le16N^2(\sum_jd_j^2)^{1/2}.
\]
Taking square roots and using \(\sigma_{N,n,j}\ge\sigma_{N,n,1}\)
proves the lemma.
\end{proof}

For part 3, we estimate the last probability in \eqref{eq:three-terms}.
\begin{proof}[Proof of Theorem~\ref{1.1}]
When \(Nl\le n(n+l)\), \eqref{eq:joint-bound} follows from
\(d\le1\). We may therefore assume \(Nl>n(n+l)\), which implies
\(N>n\).

Turning to the probability in \eqref{eq:three-terms}, recall
from Proposition~\ref{prop:projection} that, for
\(1\le j\le k\) and \(0<\varepsilon<1\),
\[
 \Prob\left(S_j-\widehat S_j\ge
       \frac j{2N}\log\frac{n}{j\varepsilon^2}\right)
 \le(C_0\varepsilon)^{j/2}.
\]
Set \(\varepsilon=\sqrt{n/j}\,e^{-Ns/j}\). For
\(s\ge CjN^{-1}\log(en/j)\), this number lies in \((0,1)\), and
\[
 (C_0\varepsilon)^{j/2}
 =\exp\left[-\frac{Ns}{2}+\frac j4\log\frac nj
                       +\frac j2\log C_0\right]
 \le e^{-Ns/4}.
\]
Since \(S_j\ge\widehat S_j\), we obtain
\[
 \Prob\left(\abs{S_j-\widehat S_j}\ge s\right)
 \le2e^{-Ns/4},\qquad
 s\ge\frac{Cj}{N}\log\frac{en}{j}.
\]
For any positive numbers \(\delta_j\ge CjN^{-1}\log(en/j)\),
we therefore have
\[
\begin{split}
 &\Prob\left(\norm{\boldsymbol S_k-\widehat{\boldsymbol S}_k}_2
                   \ge\Bigl(\sum_{j=1}^k\delta_j^2\Bigr)^{1/2}\right)\\
 &\qquad\le\sum_{j=1}^k
       \Prob\left(\abs{S_j-\widehat S_j}\ge\delta_j\right)
       \le2\sum_{j=1}^k e^{-N\delta_j/4}.
\end{split}
\]
Set
\[
 \delta_j=\frac{Cj}{N}
       \left(\log\frac{en}{j}+\log\frac Nn\right).
\]
For a sufficiently large universal \(C\),
\[
 2\sum_{j=1}^k e^{-N\delta_j/4}
 \le2\sum_{j=1}^k e^{-j}\left(\frac nN\right)^{j/2}
 \le C\left(\frac nN\right)^{1/2}.
\]
Hence
\begin{equation}\label{eq:radius}
\begin{split}
 \delta&:=\Bigl(\sum_{j=1}^k\delta_j^2\Bigr)^{1/2}
       \le\frac{Ck^{3/2}}N
                    \left(\log(en)+\log\frac Nn\right),\\
 \Prob\left(\norm{\boldsymbol S_k-\widehat{\boldsymbol S}_k}_2
                    >\delta\right)
       &\le C\left(\frac nN\right)^{1/2}.
\end{split}
\end{equation}
In conclusion, putting the three estimates
\eqref{eq:fixed-clt}, \eqref{eq:inverse}, and \eqref{eq:radius}
into \eqref{eq:three-terms}, we obtain
\begin{equation}\label{eq:bound-with-t}
 d(\Lambda_k,\Ncal(\mu_{n,\boldsymbol l},\Sigma_{n,\boldsymbol l}))
 \le{} Ck^{7/4}\left(\sqrt{\rho_N}
 +\frac{\log(en)+\log(N/n)}{\sqrt{\sigma_{N,n,1}}}\right)
 +C\left(\frac nN\right)^{1/2}.
\end{equation}
By Lemma~\ref{lem:variance},
\[
 \rho_N\le\frac{\pi^2}{2\sigma_{N,n,1}},\qquad
 \sigma_{N,n,1}\ge\frac{2Nl}{n(n+l)}.
\]
Since \(l\ge1\),
\[
\begin{split}
 \log\frac Nn
 &=\log\frac{Nl}{n(n+l)}+\log\left(1+\frac nl\right)\\
 &\le2\log\left(e+\sqrt{\frac{Nl}{n(n+l)}}\right)+\log(en).
\end{split}
\]
Also, \((n/N)^{1/2}\le\sqrt{n(n+l)/(Nl)}\). Substitution
in \eqref{eq:bound-with-t} proves \eqref{eq:joint-bound}.

For the limit under \eqref{eq:growth}, put
\[
 R=\frac{Nl}{k^{7/2}n(n+l)\log^2(en)}.
\]
If \(R\ge1\), then \(k\le n\) implies
\[
 \log\left(e+\sqrt{\frac{Nl}{n(n+l)}}\right)
 =\log(e+k^{7/4}\log(en)\sqrt R)
 \le C\log(en)+\tfrac12\log R.
\]
The bound in \eqref{eq:joint-bound} is therefore at most
\(C(1+\log R)/\sqrt R\), which tends to zero when \(R\to\infty\).
For fixed \(k\), affine invariance gives convergence of the
standardized vector to \(\Ncal(0,I_k)\).
\end{proof}

\section{Extensions of the main theorem}\label{sec:extensions}
We give the corresponding results for complex matrices and for
products whose dimensions may change between layers. The Gaussian
laws below are distributions of real vectors of log-singular values.

Gu and Liu~\cite[Theorem~1.1]{GuLiu25} obtained Gaussian limits for
the rescaled logarithms of singular values of products of rectangular
truncated unitary matrices. Their proof uses steepest-descent analysis
of the correlation kernel. Our proof uses QR decomposition, moment
estimates for independent log-Beta sums, and a projection comparison
to give explicit non-asymptotic error bounds for the joint Gaussian
approximation.

For \(\beta\in\{1,2\}\) and positive integers \(m,l\), write
\begin{align}
 \mu^{(\beta)}_{m,l}
   &=\psi(\beta m/2)-\psi(\beta(m+l)/2),\label{eq:extension-mean}\\
 \sigma^{(\beta)}_{m,l}
   &=\psi_1(\beta m/2)-\psi_1(\beta(m+l)/2).
   \label{eq:extension-variance}
\end{align}
These are the mean and variance of
\(\log\operatorname{Beta}(\beta m/2,\beta l/2)\).
The value \(\beta=1\) gives the real case and agrees with
\eqref{eq:mu-one} and \eqref{eq:sigma-one}. The value \(\beta=2\)
gives the complex case.

\subsection{A projection estimate over the two fields}
For a matrix \(Z\) over \(\mathbb F=\RR\) or \(\mathbb C\), and
an orthonormal frame \(\Theta\), use the volume norm
\[
 \norm{Z(\Theta)}^2=\det((Z\Theta)^*Z\Theta),
\]
where \(*\) denotes transpose or conjugate transpose. The next
lemma extends the projection estimate used in Section~\ref{sec:gaussian}.
Its proof follows \cite[Lemma~8.2 and Proposition~8.1]{hahn}.

\begin{lemma}[Projection error]\label{lem:extension-projection}
Let \(Z\) be a random \(p\times q\) matrix over \(\mathbb F\)
whose law is invariant under right multiplication by deterministic
orthogonal or unitary matrices. Suppose that \(\operatorname{rank}Z\ge j\)
almost surely, where \(1\le j\le\min(p,q)\). For any fixed
orthonormal \(j\)-frame \(\Theta\), any \(N\ge1\), and
\(0<\varepsilon<1\),
\begin{equation}\label{eq:extension-projection}
\begin{split}
 \Prob\left(\frac1N\sum_{h=1}^j\log s_h(Z)
       -\frac1N\log\norm{Z(\Theta)}
       \ge\frac j{2N}\log\frac{q}{j\varepsilon^2}\right)
 \le(C_0\varepsilon)^{j/2},
\end{split}
\end{equation}
where \(C_0\) is universal for both fields.
\end{lemma}

\begin{proof}
Let \(V_j\) be a frame of the first \(j\) right singular vectors.
The singular value decomposition gives
\[
 \norm{Z(\Theta)}\ge
     \left(\prod_{h=1}^j s_h(Z)\right)
                  \abs{\det(V_j^*\Theta)}.
\]
By replacing \(Z\) with \(ZQ\), where \(Q\) is an independent Haar
matrix, we can take \(V_j\) to be Haar distributed conditionally
on the original matrix. This argument also covers repeated
singular values.

Define \(D=\abs{\det(V_j^*\Theta)}\). If \(j=q\), then \(D=1\)
and the claimed probability is zero. For \(j<q\), \(V_j^*\Theta\)
has the law of the upper-left \(j\times j\) block of a \(q\times q\)
Haar matrix over \(\mathbb F\). The one-factor Gram--Schmidt
calculation in the proof of Lemma~\ref{lem:qr-beta} gives
\[
 D^2\ \stackrel{\mathrm d}{=}\ \prod_{h=1}^j B_h,\qquad
 B_h\sim\operatorname{Beta}
       \left(\frac{\beta(j-h+1)}2,\frac{\beta(q-j)}2\right),
\]
where the \(B_h\)'s are independent. Consequently,
\begin{align*}
 \EE\left[D^{-1/2}\right]
 &=\prod_{h=1}^j
   \frac{\Gamma(\beta(j-h+1)/2-1/4)
         \Gamma(\beta(q-h+1)/2)}
        {\Gamma(\beta(j-h+1)/2)
         \Gamma(\beta(q-h+1)/2-1/4)}\\
 &\le2^j\binom qj^{1/4}.
\end{align*}
Here we used the Gamma-ratio bounds
\[
 \frac12 x^{1/4}\le\frac{\Gamma(x)}{\Gamma(x-1/4)}
                    \le x^{1/4},\qquad x\ge\frac12,
\]
which follow from log-convexity of the Gamma function and its
recursion. Markov's inequality and \(\binom qj\le(eq/j)^j\) now give
\begin{align*}
 \Prob\left(D\le(j/q)^{j/2}\varepsilon^j\right)
 &\le(j/q)^{j/4}\varepsilon^{j/2}\EE\left[D^{-1/2}\right]\\
 &\le2^je^{j/4}\varepsilon^{j/2}
 \le(C_0\varepsilon)^{j/2}.
\end{align*}
The determinant bound proves \eqref{eq:extension-projection}.
\end{proof}

\subsection{Truncated Haar unitary matrices}
Let \(U_i\) be independent Haar matrices in \(\mathrm U(n+l_i)\),
and let \(A_i\) be their upper-left \(n\times n\) blocks, where
\(l_i\ge1\). Set \(X_{N,n}=A_N\cdots A_1\), and define
\(\Lambda_k\) from its singular values as in \eqref{eq:lambda}.
The complex Beta parameters also appear in
\cite[Section~4]{Reddy} and
\cite[Propositions~3.1--3.2]{Forrester15}.
For \(1\le j\le k\), put
\begin{align}
 \mu^{(2)}_{n,\boldsymbol l,j}
   &=\frac1{2N}\sum_{i=1}^N\mu^{(2)}_{n-j+1,l_i},&
 \sigma^{(2)}_{N,n,j}
   &=\sum_{i=1}^N\sigma^{(2)}_{n-j+1,l_i},
   \label{eq:unitary-parameters}\\
 \mu^{(2)}_{n,\boldsymbol l}
   &=(\mu^{(2)}_{n,\boldsymbol l,1},\ldots,
                         \mu^{(2)}_{n,\boldsymbol l,k})^{\mathsf T},&
 \Sigma^{(2)}_{n,\boldsymbol l}
   &=\frac1{4N^2}\diag(\sigma^{(2)}_{N,n,1},\ldots,
                                  \sigma^{(2)}_{N,n,k}).
   \label{eq:unitary-covariance}
\end{align}

\begin{theorem}[Truncated Haar unitary products]\label{thm:unitary-extension}
There is a universal constant \(C>0\) such that, for
\(n,N\ge1\), \(1\le k\le n\), and \(l=\min_i l_i\ge1\),
\begin{equation}\label{eq:unitary-bound}
 d\bigl(\Lambda_k,
 \Ncal(\mu^{(2)}_{n,\boldsymbol l},\Sigma^{(2)}_{n,\boldsymbol l})\bigr)
 \le{} Ck^{7/4}\sqrt{\frac{n(n+l)}{Nl}}
 \left[\log(en)+
 \log\left(e+\sqrt{\frac{Nl}{n(n+l)}}\right)\right].
\end{equation}
Under \eqref{eq:growth}, this distance tends to zero. For fixed
\(k\),
\[
 (\Sigma^{(2)}_{n,\boldsymbol l})^{-1/2}
       (\Lambda_k-\mu^{(2)}_{n,\boldsymbol l})
       \ \Longrightarrow\ \Ncal(0,I_k).
\]
For growing \(k\), convergence holds in the distance
\eqref{eq:distance} in the corresponding dimension.
\end{theorem}

\begin{proof}
Fix an orthonormal \(k\)-frame \(\Theta\) in \(\mathbb C^n\).
Starting from \(\Theta^{(0)}=\Theta\), take the complex QR decompositions
\[
 A_i\Theta^{(i-1)}=\Theta^{(i)}\mathcal R_i,\qquad 1\le i\le N,
\]
with positive real diagonal entries. In the one-factor calculation
from the proof of Lemma~\ref{lem:qr-beta}, a normalized complex
Gaussian vector gives
\[
 (\mathcal R_i)_{jj}^2\sim\operatorname{Beta}(n-j+1,l_i).
\]
The conditional law does not depend on the preceding columns or
layers, so the array of diagonal squares is independent. Let
\(T_{ij}\sim\log\operatorname{Beta}(n-j+1,l_i)\) be independent
and define \(\widehat\lambda_j\) by the volume ratios in
\eqref{eq:fixed-exponents}. Using the same QR factors for \(\Theta\)
and its leading frames gives
\begin{align*}
 \log\norm{X_{N,n}(\Theta)}
   &\stackrel{\mathrm d}{=}\frac12\sum_{i=1}^N\sum_{j=1}^k T_{ij},\\
 \widehat\lambda_j
   &\stackrel{\mathrm d}{=}\frac1{2N}\sum_{i=1}^N T_{ij},
      \qquad 1\le j\le k.
\end{align*}
The second equality holds jointly in \(j\). The mean and covariance
of \(\widehat\Lambda_k\) are given by
\eqref{eq:unitary-parameters}--\eqref{eq:unitary-covariance}.
Lemmas~\ref{lem:beta-route} and~\ref{lem:variance} apply with
\(m\) and \(l_i\) replaced by \(2m\) and \(2l_i\).
In particular,
\[
 \sigma^{(2)}_{N,n,1}\ge\frac{Nl}{n(n+l)},\qquad
 \sigma^{(2)}_{n-j+1,l_i}\le\psi_1(1)=\frac{\pi^2}{6}.
\]
The projection estimate over \(\mathbb C\) is proved in
Lemma~\ref{lem:extension-projection} above. Bentkus's theorem
applies to the real vectors of log-Beta increments. The three
estimates in the proof of Theorem~\ref{1.1} therefore give
\eqref{eq:unitary-bound} and its stated limits.
\end{proof}

In the untruncated square case, where every \(l_i=0\), the product
is unitary and all the finite-time exponents equal zero.

\subsection{Rectangular matrix products}
Let \(n_0,\ldots,n_N\) and \(l_1,\ldots,l_N\) be positive
integers satisfying
\begin{equation}\label{eq:rectangular-dimensions}
 n_{i-1}\le n_i+l_i,\qquad 1\le i\le N.
\end{equation}
For \(\mathbb F=\RR\), take independent Haar matrices \(R_i\)
in \(\mathrm O(n_i+l_i)\) and set \(\beta=1\).
For \(\mathbb F=\mathbb C\), take them in
\(\mathrm U(n_i+l_i)\) and set \(\beta=2\).
Let \(A_i\) be the upper-left \(n_i\times n_{i-1}\) block of
\(R_i\). Thus \(l_i\) counts the ambient rows outside this block,
and
\[
 A_i:\mathbb F^{n_{i-1}}\longrightarrow\mathbb F^{n_i},\qquad
 X_N=A_N\cdots A_1:\mathbb F^{n_0}\longrightarrow\mathbb F^{n_N}.
\]
Fix \(1\le k\le\min_{0\le i\le N}n_i\). Define
\(\lambda_j=N^{-1}\log s_j(X_N)\), \(1\le j\le k\), using the
positive singular values, and put
\(\Lambda_k=(\lambda_1,\ldots,\lambda_k)^{\mathsf T}\).
Write \(\boldsymbol n=(n_0,\ldots,n_N)\) and
\(\boldsymbol l=(l_1,\ldots,l_N)\). Set
\begin{align}
 \mu^{(\beta)}_{\boldsymbol n,\boldsymbol l,j}
   &=\frac1{2N}\sum_{i=1}^N\mu^{(\beta)}_{n_i-j+1,l_i},
       \label{eq:rectangular-mean}\\
 \sigma^{(\beta)}_{N,\boldsymbol n,j}
   &=\sum_{i=1}^N\sigma^{(\beta)}_{n_i-j+1,l_i},
       \label{eq:rectangular-variance}\\
 \mu^{(\beta)}_{\boldsymbol n,\boldsymbol l}
   &=(\mu^{(\beta)}_{\boldsymbol n,\boldsymbol l,1},\ldots,
                  \mu^{(\beta)}_{\boldsymbol n,\boldsymbol l,k})^{\mathsf T},
       \notag\\
 \Sigma^{(\beta)}_{\boldsymbol n,\boldsymbol l}
   &=\frac1{4N^2}\diag(\sigma^{(\beta)}_{N,\boldsymbol n,1},\ldots,
                                   \sigma^{(\beta)}_{N,\boldsymbol n,k}).
       \label{eq:rectangular-covariance}
\end{align}

\begin{theorem}[Rectangular products]\label{thm:rectangular-extension}
Under \eqref{eq:rectangular-dimensions}, the product has rank
\(\min_{0\le i\le N}n_i\) almost surely. There is a universal
constant \(C>0\), valid over both fields, such that
\begin{equation}\label{eq:rectangular-bound}
 d\bigl(\Lambda_k,
 \Ncal(\mu^{(\beta)}_{\boldsymbol n,\boldsymbol l},
 \Sigma^{(\beta)}_{\boldsymbol n,\boldsymbol l})\bigr)
 \le{}\frac{Ck^{7/4}}{\sqrt{\sigma^{(\beta)}_{N,\boldsymbol n,1}}}
 \left[\log(en_0)+
 \log\left(e+\sqrt{\sigma^{(\beta)}_{N,\boldsymbol n,1}}\right)\right].
\end{equation}
In particular, the distance tends to zero if
\begin{equation}\label{eq:rectangular-growth}
 \frac{\sigma^{(\beta)}_{N,\boldsymbol n,1}}
      {k^{7/2}\log^2(en_0)}\longrightarrow\infty.
\end{equation}
All the dimensions and truncations may vary with \(N\).
For fixed \(k\), this gives
\[
 (\Sigma^{(\beta)}_{\boldsymbol n,\boldsymbol l})^{-1/2}
       (\Lambda_k-\mu^{(\beta)}_{\boldsymbol n,\boldsymbol l})
       \ \Longrightarrow\ \Ncal(0,I_k).
\]
For growing \(k\), the conclusion is convergence in the distance
\eqref{eq:distance}.
\end{theorem}

\begin{proof}
Start from a fixed orthonormal \(k\)-frame \(\Theta^{(0)}=\Theta\) in
\(\mathbb F^{n_0}\). Write \(\Theta_j\) for its first \(j\)
columns, and take the thin QR decompositions
\[
 A_i\Theta^{(i-1)}=\Theta^{(i)}\mathcal R_i.
\]
The frame \(\Theta^{(i)}\) lies in \(\mathbb F^{n_i}\), and
\(\mathcal R_i\) is a \(k\times k\) upper triangular matrix with
positive diagonal. Conditioned on the preceding layers, right
invariance reduces the calculation to the first \(k\) columns
of the ambient Haar matrix. At the \(j\)th Gram--Schmidt step,
conditional on the preceding Haar columns, the next column is
uniform on a sphere in dimension \(n_i+l_i-j+1\). The residual
subspace has dimension \(n_i-j+1\), with an orthogonal complement
of dimension \(l_i\) in that space. The squared residual has the
law of \(G/(G+H)\), where \(G,H\) are independent Gamma variables
with shapes \(\beta(n_i-j+1)/2,\beta l_i/2\) and rate one. Thus
\begin{equation}\label{eq:rectangular-beta}
 (\mathcal R_i)_{jj}^2
       \sim\operatorname{Beta}
          \left(\frac{\beta(n_i-j+1)}2,\frac{\beta l_i}2\right).
\end{equation}
The conditional law does not depend on the preceding columns or
layers, so the diagonal squares are independent. All the residuals
are positive almost surely. Taking \(k=\min_i n_i\) also proves
the rank assertion.

Set \(\widehat S_0=0\), and for \(1\le j\le k\) let
\[
 \widehat S_j=\frac1N\log\norm{X_N(\Theta_j)},\qquad
 \widehat\lambda_j=\widehat S_j-\widehat S_{j-1}.
\]
Let \(T_{ij}\) be independent with the laws of the logarithms of
the Beta variables in \eqref{eq:rectangular-beta}. As in
Lemma~\ref{lem:qr-beta}, using the same QR factors for \(\Theta\)
and its leading frames gives
\begin{align*}
 \log\norm{X_N(\Theta)}
   &\stackrel{\mathrm d}{=}\frac12\sum_{i=1}^N\sum_{j=1}^k T_{ij},\\
 \widehat\lambda_j
   &\stackrel{\mathrm d}{=}\frac1{2N}\sum_{i=1}^N T_{ij},
      \qquad 1\le j\le k.
\end{align*}
The second equality holds jointly in \(j\).
The vector \(\widehat\Lambda_k\) has independent coordinates,
mean \(\mu^{(\beta)}_{\boldsymbol n,\boldsymbol l}\), and covariance
\(\Sigma^{(\beta)}_{\boldsymbol n,\boldsymbol l}\).
Use \(S_j=\sum_{h=1}^j\lambda_h\) and form
\(\boldsymbol S_k,\widehat{\boldsymbol S}_k\) as in
Section~\ref{sec:gaussian}.

Lemma~\ref{lem:variance}, with integer parameters
\(\beta(n_i-j+1)\) and \(\beta l_i\), gives
\[
 \EE\left[\abs{T_{ij}-\EE[T_{ij}]}^3\right]
     \le C\bigl(\sigma^{(\beta)}_{n_i-j+1,l_i}\bigr)^{3/2}.
\]
Also,
\[
 \sigma^{(\beta)}_{n_i-j+1,l_i}\le\psi_1(\beta/2)\le\frac{\pi^2}{2},
 \qquad
 \sigma^{(\beta)}_{N,\boldsymbol n,j}
             \ge\sigma^{(\beta)}_{N,\boldsymbol n,1}.
\]
The proof of Lemma~\ref{lem:fixed-clt} therefore yields
\[
 d\bigl(\widehat\Lambda_k,
       \Ncal(\mu^{(\beta)}_{\boldsymbol n,\boldsymbol l},
             \Sigma^{(\beta)}_{\boldsymbol n,\boldsymbol l})\bigr)
 \le\frac{Ck^{7/4}}{\sqrt{\sigma^{(\beta)}_{N,\boldsymbol n,1}}}.
\]
For the partial sums, use the same lower triangular matrix \(T\)
as in Section~\ref{sec:gaussian}. Lemma~\ref{lem:inverse} gives
\[
 \norm{(T\Sigma^{(\beta)}_{\boldsymbol n,\boldsymbol l}
                       T^{\mathsf T})^{-1}}_{\mathrm{HS}}^{1/2}
 \le\frac{4Nk^{1/4}}{\sqrt{\sigma^{(\beta)}_{N,\boldsymbol n,1}}}.
\]

The product \(X_N\) is right invariant over \(\mathbb F^{n_0}\).
Apply Lemma~\ref{lem:extension-projection} with \(q=n_0\).
For \(t\ge1\), choose
\[
 \delta_j=\frac{Cj}{N}\left(\log\frac{en_0}{j}+t\right),\qquad
 \delta=\left(\sum_{j=1}^k\delta_j^2\right)^{1/2}
       \le\frac{Ck^{3/2}}N\bigl(\log(en_0)+t\bigr).
\]
The scalar tail estimate used in the proof of \eqref{eq:radius}
now holds with \(n\) replaced by \(n_0\). Summing the coordinate
tail probabilities gives
\[
 \Prob\left(\norm{\boldsymbol S_k-\widehat{\boldsymbol S}_k}_2>\delta\right)
 \le2\sum_{j=1}^k e^{-jt}\le Ce^{-t}.
\]
Proposition~\ref{6.3} and affine invariance then imply
\[
 d\bigl(\Lambda_k,
 \Ncal(\mu^{(\beta)}_{\boldsymbol n,\boldsymbol l},
 \Sigma^{(\beta)}_{\boldsymbol n,\boldsymbol l})\bigr)
 \le
 \frac{Ck^{7/4}(\log(en_0)+t)}
 {\sqrt{\sigma^{(\beta)}_{N,\boldsymbol n,1}}}+Ce^{-t}.
\]
Take \(t=\log(e+\sqrt{\sigma^{(\beta)}_{N,\boldsymbol n,1}})\).
The last term is absorbed by the first, proving
\eqref{eq:rectangular-bound}. Since \(k\le n_0\), the same
calculation as at the end of the proof of Theorem~\ref{1.1}
proves convergence under \eqref{eq:rectangular-growth}.
Affine invariance gives the stated normalized limit.
\end{proof}

\begin{corollary}[Dimension conditions]\label{cor:rectangular-growth}
The Gaussian distance in \eqref{eq:rectangular-bound} tends to
zero if
\begin{equation}\label{eq:rectangular-profile-growth}
 \frac{\displaystyle\sum_{i=1}^N
                    \frac{l_i}{n_i(n_i+l_i)}}
      {k^{7/2}\log^2(en_0)}\longrightarrow\infty.
\end{equation}
In particular, put \(M=\max_{0\le i\le N}n_i\) and
\(l=\min_{1\le i\le N}l_i\). A sufficient condition is
\begin{equation}\label{eq:rectangular-uniform-growth}
 \frac{Nl}{k^{7/2}M(M+l)\log^2(eM)}\longrightarrow\infty.
\end{equation}
For fixed \(k\), it suffices that
\[
 N\gg\left(M+\frac{M^2}{l}\right)\log^2(eM).
\]
Thus \(N\gg M^2\log^2(eM)\) suffices for fixed \(l\), and
\(N\gg M\log^2(eM)\) suffices when \(l\ge cM\) for a fixed
\(c>0\).
\end{corollary}

\begin{proof}
The variance comparison in \eqref{eq:variance-comparison}, applied
to the integer parameters \(\beta n_i\) and \(\beta l_i\), gives
\[
 \sigma^{(\beta)}_{N,\boldsymbol n,1}
 \ge\frac2\beta\sum_{i=1}^N\frac{l_i}{n_i(n_i+l_i)}
 \ge\frac{2Nl}{\beta M(M+l)}.
\]
Use \(\log(en_0)\le\log(eM)\) in
\eqref{eq:rectangular-growth}.
\end{proof}

\section*{Acknowledgements}
I am grateful to  Dang-Zheng Liu for his guidance.
I would also like to thank Yandong Gu, Guangyi Zou, and Ruohan Geng
for their useful advice.

\end{document}